\documentclass[10pt,reqno]{amsart}
\usepackage{amsmath}
\usepackage{amssymb}
\usepackage{amsthm}
\usepackage{mathrsfs}

\newlength{\defbaselineskip}
\newcommand{\setlinespacing}[1]%
           {\setlength{\baselineskip}{#1 \defbaselineskip}}

\numberwithin{equation}{section}

\newtheorem{thm}{Theorem}[section]

\newtheorem{lem}[thm]{Lemma}

\theoremstyle{definition}

\theoremstyle{remark}

\numberwithin{equation}{section}

\begin{document}
\title[Weighted estimates involving singular weights]{On weighted estimates involving singular weights for the Dirac equation}

\author{Seongyeon Kim and Ihyeok Seo}

\subjclass[2010]{Primary: 35B45; Secondary: 35Q41, 42B35}
\keywords{weighted estimates, Dirac equation, Kerman-Sawyer class}

\address{Department of Mathematics Education, Jeonju University, Jeonju 55069, Republic of Korea}
\email{sy\_kim@jj.ac.kr}

\address{Department of Mathematics, Sungkyunkwan University, Suwon 16419, Republic of Korea}
\email{ihseo@skku.edu}

\begin{abstract}
We obtain weighted $L^2$ estimates for solutions to the Dirac equation
with weights in the Kerman-Sawyer class. This class has been introduced as a necessary and sufficient condition for the boundedness of fractional integral operators in weighted $L^2$ spaces. 
\end{abstract}

\maketitle

\section{Introduction}		
In relativistic quantum mechanics, the state of a free electron with mass $m > 0$ is described by a wave function $u(x,t):\mathbb{R}^{3+1} \rightarrow \mathbb{C}^{4}$ governed by the Dirac equation:
\begin{equation}\label{DE}
\begin{cases}
-i\partial_{t}u+\mathcal{D}u+m\beta u=0,\\
u(x,0)=f(x),
\end{cases}
\end{equation}
where the Dirac operator $\mathcal{D}$ is a first-order differential operator defined by
\begin{equation*}
\mathcal{D} = -i \sum_{j=1}^{3} \alpha_j \partial_j = -i \alpha \cdot \nabla,
\end{equation*}
with $4 \times 4$ complex matrices $\alpha_j$ and $\beta$ (called the Dirac matrices) given by
\begin{equation*}
\alpha_{j}=
\begin{pmatrix}
0 & \sigma_{j}\\
\sigma_{j} & 0
\end{pmatrix},
\quad
\beta=
\begin{pmatrix}
I_{2} & 0\\
0 & -I_{2}
\end{pmatrix},
\end{equation*}
constructed from the Pauli matrices:
\begin{equation*}
\sigma_{1}=
\begin{pmatrix}
0 & 1\\
1 & 0
\end{pmatrix},
\quad
\sigma_{2}=
\begin{pmatrix}
0 & -i\\
i & 0
\end{pmatrix},
\quad
\sigma_{3}=
\begin{pmatrix}
1 & 0\\
0 & -1
\end{pmatrix}.
\end{equation*}
These matrices satisfy the following identities for $j,k=1,2,3$:
\begin{equation}\label{matid}
\beta^{2} = \alpha_{j}^{2} = I_{4}, \quad \alpha_{j}\beta + \beta\alpha_{j} = 0, \quad \alpha_{j}\alpha_{k} + \alpha_{k}\alpha_{j} = 2\delta_{jk}I_{4}.
\end{equation}
See, for instance, \cite{Th} for a more detailed introduction to the Dirac equation.

In this paper, we are concerned with a weighted $L^2$ estimate for solutions to the Dirac equation involving singular weights $w \ge 0$. In particular, we focus on weights in the Kerman–Sawyer class $\mathcal{KS}$, defined by
\[
\|w\|_{\mathcal{KS}} := \sup_Q \left( \int_Q w(x) dx \right)^{-1} \int_Q \int_Q \frac{w(x)w(y)}{|x-y|} dxdy < \infty,
\]
where the supremum is taken over all dyadic cubes $Q \subset \mathbb{R}^3$.
This class, introduced by Kerman and Sawyer \cite{KeS}, provides a necessary and sufficient condition on weights for the validity of weighted $L^2$ estimates for fractional integral operators.

It is closely related to the Kato ($\mathcal{K}$) and Rollnik ($\mathcal{R}$) classes, which play a fundamental role in spectral and scattering theory (see \cite{K, Si}). Their definitions are given by
\begin{equation*}
w \in \mathcal{K} \quad \Leftrightarrow \quad \|w\|_{\mathcal{K}} := \sup_{x \in \mathbb{R}^3} \int_{\mathbb{R}^3} \frac{w(y)}{|x-y|} dy < \infty
\end{equation*}
and
\begin{equation*}
w \in \mathcal{R} \quad \Leftrightarrow \quad \|w\|_{\mathcal{R}} := \left( \int_{\mathbb{R}^3} \int_{\mathbb{R}^3} \frac{w(x)w(y)}{|x - y|^2} dx dy \right)^{1/2} < \infty.
\end{equation*}
In particular, it is known that $\mathcal{K} \subset \mathcal{KS}$ and $\mathcal{R} \subset \mathcal{KS}$.
The class $\mathcal{KS}$ also contains the Fefferman–Phong class $\mathcal{F}^p$, defined for $1 < p <3/2$ as
\[
\|w\|_{\mathcal{F}^p} := \sup_{x \in \mathbb{R}^3,\, r > 0} r^{2 - 3/p} \left( \int_{|y - x| < r} w(y)^p dy \right)^{1/p} < \infty.
\]
Note that $|x|^{-2} \in L^{3/2,\infty} \subsetneq \mathcal{F}^p\subset \mathcal{KS}$
for $1 < p <3/2$.

Let $L^2(w)$ denote the weighted $L^2$ space equipped with the norm
\[
\|f\|_{L^2(w)} := \left( \int |f(x)|^2 w(x)\, dx \right)^{1/2}.
\]
Our main result is the following:

\begin{thm}\label{thm}
Let $w \in \mathcal{KS}$ and $f\in H^{1/2}$. Then the solution $u$ to the Dirac equation \eqref{DE} satisfies 
\begin{equation}\label{con}
u \in C([0,\infty); H^{1/2})
\end{equation}
and
\begin{equation*}
\|u\|_{L^2_{x,t}(w)} \lesssim \|w\|_{\mathcal{KS}}^{1/2} \|f\|_{H^{1/2}}.
\end{equation*}
\end{thm}

While such weighted estimates have been extensively studied for the Schr\"odinger equation
(\cite{RV,BBCRV,S,S2}), the fundamental equation of non-relativistic quantum mechanics,
they have---to the best of our knowledge---not been addressed for its relativistic counterpart,
the Dirac equation, in general function spaces beyond certain specific power-type weights such as
\[
w_\varepsilon(x)=\bigl(|x|^{1/2-\varepsilon}+|x|\bigr)^{-2},
\qquad 0<\varepsilon\ll1,
\]
(see, e.g., \cite{DAF}).

It is worth noting that this specific weight belongs to neither the Kato nor the Rollnik class,
while it does belong to the Kerman--Sawyer class. Indeed, since
$w_\varepsilon(x)\sim |x|^{-1+2\varepsilon}$ near the origin and
$w_\varepsilon(x)\sim |x|^{-2}$ at infinity, and $w_\varepsilon$ is radial and decreasing,
it is readily verified from the definition, by considering balls centered at the origin, that
$w_\varepsilon\in\mathcal{F}^p\subset\mathcal{KS}$ for $1<p<3/2$.
On the other hand, $w_\varepsilon\notin\mathcal{K}$, as can be seen by taking $x=0$
in the definition of the Kato norm. Moreover, $w_\varepsilon\notin\mathcal{R}$.
Indeed, for the dyadic annulus
$A_R=\{x\in \mathbb R^3 :R<|x|<2R\}$ with $R\ge1$, we have
$w_\varepsilon(x)\sim w_\varepsilon(y)\sim R^{-2}$,
$|x-y|\lesssim R$, and $|A_R|\sim R^3$, which yields divergence of the Rollnik norm.

Thus, the Kerman--Sawyer framework adopted here not only naturally encompasses this
previously known power-type weight, but also allows us to treat all singular weights
in the Kato and Rollnik classes, which have not previously
been covered in the context of the Dirac equation.

\medskip

Throughout the paper, the symbol $C$ denotes a positive constant that may vary from line to line. We use the notation $A \lesssim B$ to mean $A \leq C B$ for some unspecified constant $C > 0$.

\section{Proof of Theorem \ref{thm}}\label{sec2}
For simplicity, we rescale the variables by setting $m = 1$.
Now we consider the Dirac equation
\begin{equation}\label{inDE}
\begin{cases}
-i\partial_{t}u + \mathcal{D}u + \beta u = 0, \\
u(x,0) = f(x).
\end{cases}
\end{equation}

To proceed, we define the projection operators $\Pi_{\pm}(D)$ by
\[
\Pi_{\pm}(D) = \frac{1}{2} \left(I \pm \frac{1}{\langle D \rangle}(\mathcal{D} + \beta)\right),
\]
where $D = -i\nabla$ and $\langle D \rangle = \sqrt{1 - \Delta}$. 
Then,
\[
\mathcal{D}+ \beta = \langle D \rangle (\Pi_+(D) - \Pi_-(D)),
\]
and using the matrix identities in \eqref{matid}, one observes
$$(\mathcal{D}+\beta)^2=I_4 -\Delta$$ 
which leads to
\begin{equation}\label{orth}
\Pi_{\pm}(D)^2 = \Pi_{\pm}(D) \quad\text{and}\quad \Pi_{+}(D)\Pi_{-}(D) = \Pi_{-}(D)\Pi_{+}(D) = 0.
\end{equation}

We denote $u_{\pm} = \Pi_{\pm}(D)u$ and $f_{\pm} = \Pi_{\pm}(D)f$.
Applying $\Pi_{\pm}(D)$ to equation \eqref{inDE} yields the following system:
\begin{equation*}
\begin{cases}
-i\partial_t u_{\pm} \pm \langle D \rangle u_{\pm} = 0, \\
u_{\pm}(x,0) = f_{\pm},
\end{cases}
\end{equation*}
whose solution is given by
\begin{equation*}
u_{\pm}(x,t) = e^{\mp it \langle D \rangle}f_{\pm}(x)
\end{equation*}
where
\[
e^{\mp it \langle D \rangle}f_{\pm}(x) = \frac{1}{(2\pi)^3} \int_{\mathbb{R}^{3}} e^{i(x \cdot \xi \mp t\langle \xi \rangle)} \widehat{f_{\pm}}(\xi)\, d\xi.
\]

If we define $e^{-it(\mathcal{D}+\beta)}$ by
$$e^{-it(\mathcal{D}+\beta)}f=e^{- it \langle D \rangle}f_{+}(x)+
e^{it \langle D \rangle}f_{-}(x),$$
then the solution to \eqref{inDE} is given by
\begin{equation*}
u(x,t) = e^{-it (\mathcal{D}+\beta)}f(x)
\end{equation*}
since $u=u_+ + u_-$.
Hence, we only need to prove 
\begin{equation*}
\sup_{t\in\mathbb{R}}\|e^{-it (\mathcal{D}+\beta)}f\|_{H^{1/2}}\lesssim\|f\|_{H^{1/2}},
\end{equation*}
which implies \eqref{con}, and
\begin{equation*}
\|e^{-it (\mathcal{D}+\beta)}f\|_{L_{x,t}^2(w)}\lesssim\|w\|_{\mathcal{KS}}^{1/2}\|f\|_{H^{1/2}}.
\end{equation*}
Since $\|f\|_{H^{1/2}} \sim \|f_+\|_{H^{1/2}} + \|f_-\|_{H^{1/2}}$ by \eqref{orth},
these estimates follow now immediately from the following estimates for the scalar flows $e^{\pm it\langle D \rangle}$:
\begin{equation*}
\sup_{t\in\mathbb{R}}\|e^{\pm it \langle D \rangle}f\|_{H^{1/2}}\lesssim\|f\|_{H^{1/2}}
\end{equation*}
and
\begin{equation*}
\|e^{\pm it\langle D \rangle}f\|_{L_{x,t}^2(w)}\lesssim\|w\|_{\mathcal{KS}}^{1/2}\|f\|_{H^{1/2}}.
\end{equation*}

The first estimate is obtained directly from using the Plancherel theorem, whereas the second estimate will be addressed separately below. 

\begin{lem}\label{KGWL}
Let $w\in \mathcal{KS}$ and $f\in H^{1/2}$. Then we have
\begin{equation}\label{homo}
\big\| e^{\pm it\langle D \rangle}f\big\|_{L_{x,t}^2 (w)} \lesssim \|w\|_{\mathcal{KS}}^{1/2} \|\langle D \rangle^{1/2}f\|_{L^2}.
\end{equation}
\end{lem}

\begin{proof}
By the change of variables $t \mapsto -t$, it is enough to prove \eqref{homo} for $e^{it\langle D \rangle}$ only.
Using polar coordinates $\xi \rightarrow r\sigma$ and a change of variables $\sqrt{1+r^2} \rightarrow r$, we have
\begin{align*}
e^{it\langle D \rangle}f(x) &=  \int_{\mathbb{R}^3} e^{ix\cdot\xi} e^{it\sqrt{1+|\xi|^{2}}} \widehat{f}(\xi)\,d\xi\\
&= \int_{0}^{\infty} \int_{\mathbb{S}_{r}^{2}} e^{ix\cdot r\sigma} e^{it\sqrt{1+r^{2}}} \widehat{f}(r\sigma) \,d\sigma_{r}dr\\
&=\int_{1}^{\infty} \int_{\mathbb{S}_{\sqrt{r^2-1}}^{2}} e^{ix\cdot\sqrt{r^2-1}\sigma} e^{itr}\widehat{f}(\sqrt{r^2-1}\sigma)
\frac{r}{\sqrt{r^2-1}} \,d\sigma_{\sqrt{r^2-1}}\,dr\\
&=\int_{-\infty}^\infty e^{itr}\chi_{(1,\infty)}(r) \frac{r}{\sqrt{r^2-1}} \big(\widehat{f}d\sigma_{\sqrt{r^2-1}}\big)^{\wedge}(-x)\,dr.
\end{align*}

By applying Plancherel's theorem in the $t$-variable and then using a change of variables $ \sqrt{r^2-1} \rightarrow r$ again, we see
\begin{align}\label{homopf1}
\big\| e^{it\langle D \rangle}f \big\|_{L_{x,t}^{2}(w)}^{2} &=\int_{\mathbb{R}^3}\int_{1}^{\infty}
\bigg|\frac{r}{\sqrt{r^2-1}} \big(\widehat{f}d\sigma_{\sqrt{r^2-1}}\big)^{\wedge}(-x)\bigg|^{2}w(x)\,drdx\nonumber\\
&=\int_{0}^{\infty}  \int_{\mathbb{R}^3}\frac{\sqrt{1+r^2}}{r}\Big|\widehat{\widehat{f}d\sigma_{r}}(-x)\Big|^{2}w(x)dxdr\nonumber\\
&=\int_{0}^{\infty} \frac{\sqrt{1+r^2}}{r}
\Big\| \widehat{\widehat{f}d\sigma_{r}}(-x)\Big\|_{L_x^2(w)}^2dr.
\end{align}

Finally, we apply the following weighted $L^2$ restriction estimate for the Fourier transform:
\begin{equation*}
\big\|\widehat{fd\sigma_r}\big\|_{L^{2}(w)} \lesssim r^{1/2}\|w\|_{\mathcal{KS}}^{1/2}\|f\|_{L^{2}(\mathbb{S}_r^{2})},
\end{equation*}
 which can be found in Corollary 3.2 in \cite{S} (see also \cite{CR,CS,BBCRV}).
Applying this to the right-hand side of \eqref{homopf1}, we have
\begin{align*}
\big\|e^{it\langle D \rangle}f\big\|_{L_{x,t}^2(w)}^{2}
&\lesssim \int_{0}^{\infty} \sqrt{1+r^2} \|w\|_{\mathcal{KS}}\int_{\mathbb{S}_r^{2}} |\widehat{f}(r\sigma)|^{2}d\sigma_{r}dr\\
&= \|w\|_{\mathcal{KS}} \int_{0}^{\infty}\int_{\mathbb{S}_r^{2}}|(1+r^2)^{1/4} \widehat{f}(r\sigma)|^{2}d\sigma_r dr\\
&= \|w\|_{\mathcal{KS}} \|\langle D \rangle^{1/2}f\|_{L^2}^2
\end{align*}
as desired.
\end{proof}


\begin{thebibliography}{99}

\bibitem{BBCRV} J. A. Barcel\`o, J. M. Bennett, A. Carbery, A. Ruiz and M. C. Vilela, \textit{A note on weighted estimates for the Schr\"odinger operator}, Rev. Mat. Complut. 21 (2008), 481-488.


\bibitem{CS} S. Chanillo and E. Sawyer, \textit{Unique continuation for $\Delta + v$ and the C. Fefferman-Phong class},
Trans. Amer. Math. Soc. 318 (1990), 275-300.

\bibitem{CR} F. Chiarenza and A. Ruiz, \textit{Uniform $L^2$-weighted Sobolev inequalities}, Proc. Amer. Math. Soc. 112 (1991), 53-64.


\bibitem{DAF} P. D'Ancona and L. Fanelli, \textit{Strichartz and smoothing estimates for dispersive equations with magnetic potentials},
Comm. Partial Differential Equations 33 (2008), 1082-1112.

\bibitem{K} T. Kato, \textit{Wave operators and similarity for some non-selfadjoint operators}, Math. Ann. 162 (1966), 258-279.

\bibitem{KeS} R. Kerman and E. Sawyer, \textit{The trace inequality and eigenvalue estimates for Schr\"odinger operators},
Ann. Inst. Fourier (Grenoble) 36 (1986), 207-228.


\bibitem{RV} A. Ruiz and L. Vega, \textit{Local regularity of solutions to wave equations with time-dependent potentials}, Duke Math. J. 76 (1994), 913-940.

\bibitem{S} I. Seo, \textit{From resolvent estimates to unique continuation for the Schr\"odinger equation},
Trans. Amer. Math. Soc. 368 (2016), 8755-8784.

\bibitem{S2} I. Seo, \textit{A note on the Schr\"odinger smoothing effect},
Math. Nachr. 292 (2019), 2481–2487.

\bibitem{Si} B. Simon, \textit{Schr\"odinger semigroups}, Bull. Amer. Math. Soc. 7 (1982), 447-526.


\bibitem{Th} B. Thaller, \textit{The Dirac equation}, Texts and Monographs in Physics. Springer-Verlag, Berlin, 1992.



\end{thebibliography}
\end{document}